\documentclass[12pt]{article}
\usepackage{amsmath, amsthm}

\newtheorem{theorem}{Theorem}
\newtheorem{corollary}[theorem]{Corollary}
\newtheorem{lemma}[theorem]{Lemma}
\def\pd#1{\partial_{#1}}
\def\indicator#1{{\bf 1}_{#1}}
\def\rank#1{{\rm rank}\ #1}
\def\Prob{{\bf P}}
\def\x{x} 
\def\y{y} 
\def\t{t} 
\def\d{d} 

\title{
Calculation of orthant probabilities by the holonomic gradient method
}
\author{Tamio Koyama\thanks{Department of Mathematics, Kobe University}\ \  and Akimichi Takemura\thanks{Graduate School of Information Science and Technology, University of Tokyo}\ \thanks{JST CREST}}

\begin{document}

\maketitle

\begin{abstract}
We apply the holonomic gradient method (HGM) introduced by \cite{n3ost2}
to the calculation of orthant probabilities of multivariate normal distribution.
The holonomic gradient method applied to orthant probabilities is found to be
a variant of Plackett's recurrence relation (\cite{plackett}).  However 
an implementation of the method yields recurrence relations
more suitable for numerical computation than Plackett's recurrence relation.
We derive some theoretical results on 
the holonomic system for the orthant probabilities.
These results show that multivariate normal orthant probabilities possess some
remarkable properties from the viewpoint of holonomic systems.
Finally we show that numerical performance
of our method is comparable or superior compared to existing methods.
\end{abstract}

\section{Introduction}\label{sec:intro}
The holonomic gradient method (HGM) introduced by \\
Nakayama et al.\ \cite{n3ost2}
is a new method of numerical calculation which utilizes algebraic properties 
of differential equations. 
This method has many applications in statistics. 
For example, 
an application to the evaluation of the exact distribution function of
the largest root of a Wishart matrix is introduced in \cite{hnt2} 
and an application to the maximum likelihood estimation for the
Fisher-Bingham distribution on the $d$-dimensional sphere is introduced in \cite{kn2t}.  These applications greatly expand the scope of the field of algebraic statistics.

In this paper, we utilize the holonomic gradient method for 
an accurate evaluation of the  orthant probability
\begin{equation}\label{eq:orth_prob}
\Phi(\Sigma, \mu)={\bf P}(X_1 \geq 0, \dots, X_d\geq 0),
\end{equation}
where the $d$-dimensional random vector $X=(X_1,\dots, X_\d) \in {\bf R}^\d$ is normally
distributed with mean $\mu$ and covariance matrix $\Sigma$, i.e., 
$X \sim N(\mu, \Sigma)$.
Since evaluation of  the orthant probability has important applications in
statistical practice, there are many studies about it.
Genz introduced a method to calculate the orthant probability utilizing
the quasi Monte-Carlo method in \cite{Genz1}.
Miwa, Hayter and Kuriki proposed a recursive
integration algorithm to evaluate the orthant probability in
\cite{Miwa}. 

When the mean vector $\mu$ is equal to zero, 
the orthant probability can be interpreted as the area of 
a $d-1$ dimensional spherical simplex (cf.\ \cite{aomoto}).
In \cite{schlafli}, Schl\"{a}fli gave a classical differential recurrence 
formula for $\Phi(\Sigma, 0)$.
Plackett generalized Schl\"{a}fli's result 
and gave a recurrence formula for $\Phi(\Sigma, \mu)$
in \cite{plackett}. 
He evaluated orthant probabilities 
by a recursive integration utilizing his formula.
Gassmann implemented the  Plackett's method in the case of higher
dimensions in \cite{Gassmann}.
Their method is a recursive integration 
based on a differential recurrence formula, whereas
the holonomic gradient method utilizes differential equations.
Plackett's recurrence formula is not suitable for 
the holonomic gradient method and in this paper 
we give a new recurrence formula, which is more natural from 
the viewpoint of holonomic gradient method.

Let 
$
\y = \Sigma^{-1}\mu
$
and 
$\x =-\frac{1}{2}\Sigma^{-1}$.
We denote the $i$-th element of $\y$ by $\y_i$ and
the $(i,j)$ element of $\x$ by $\x_{ij}$.
By this transformation of the parameters,
the orthant probability in $(\ref{eq:orth_prob})$ can be written as
\[
\left(-\pi\right)^{-d/2}
\det (x)^{1/2}
\exp\left(
-\frac{1}{2}\y^t\x^{-1}\y
\right)g(\x, \y), 
\]
where 
\begin{equation}
\label{eq:g}
g(\x, \y) =
\int_0^\infty\dots\int_0^\infty
 \exp\left(\sum_{i,j=1}^\d\x_{ij}\t_i\t_j+\sum_{i=1}^\d\y_i\t_i\right) d\t \qquad
(dt = dt_1 \dots dt_d).
\end{equation}
In order to evaluate the orthant probability, it is enough to evaluate
$g(\x,\y)$. 

To apply the holonomic gradient method, we need an explicit 
form of a Pfaffian system (\cite[Section 2]{kn2t}) associated with $g(\x,\y)$.
The Pfaffian system can be obtained from a holonomic system for $g(\x,\y)$
(see \cite{n3ost2}).
For a given $\d$, we can obtain a holonomic system for $g(\x,\y)$
by applying the algorithms introduced in \cite{Oaku3}
with computer algebra systems.
However, for general $\d$, these algorithms can not be applied
and we need theoretical considerations.
In this paper, we give a holonomic system for the function
$g(\x,\y)$ and construct a Pfaffian system from the holonomic
system. 

Note that the integral $g(\x,\y)$ satisfies 
an incomplete $A$-hypergeometric system (\cite{NT})
when the integration domain is not the orthant but a polytope.

The organization of this paper is as follows.
In Section $\ref{sec:expectation}$, we describe a holonomic system for
an integral
\begin{equation}
 \label{eq:general_g}
g(\x, \y)
 =\int_{{\bf R}^\d}
 f(t)\exp\left(\sum_{i,j=1}^\d\x_{ij}\t_i\t_j+\sum_{i=1}^\d\y_i\t_i\right) 
 dt, 
\end{equation}
where $f(t)$ is a holonomic function or a holonomic distribution.
In Section $\ref{sec:orthant}$, we construct a holonomic system for the
function in $(\ref{eq:g})$  by utilizing the result in Section
$\ref{sec:expectation}$. Then we construct a Pfaffian system from the
holonomic system.
Finally in Section $\ref{sec:numerical}$ we describe numerical experiments of
the holonomic gradient method.

\section{Holonomic system associated with the expectation under  multivariate normal distributions}\label{sec:expectation}
In this section, we consider the holonomic ideal which annihilates
the integral $(\ref{eq:general_g})$, which is the expectation under
multivariate normal distributions
(for the definition of the holonomic ideal, see \cite{SST}).
 We develop a general theory, where
$f(t)$ in \eqref{eq:general_g} is a smooth function or a distribution in the sense of  Schwartz (\cite{Schwartz}).
This theory is a generalization of the result introduced in \cite{koyama}.
In Section  $\ref{sec:orthant}$ we will specialize $f(t)$ to be the indicator function of the positive orthant.
Note that the results of this section can be applied to various problems of the multivariate
normal distribution theory, other than the orthant probability.

We denote  the ring of differential operators in $x$ with polynomial coefficient 
by $D={\bf C}\langle x_1,\dots,x_n, \pd{x_1}, \dots, \pd{x_n} \rangle$. 
The operator $\partial/\partial x_i $ is denoted by $\pd{x_i}$. 
We frequently use the following rings:
\begin{align}
D_{\x\y\t} &:= 
{\bf C}\langle \x_{ij}, \y_k, \t_k, \pd{\x_{ij}}, \pd{\y_k}, \pd{t_k}: 1\leq i\leq j \leq \d, 1\leq k \leq \d \rangle, \nonumber \\
D_{\x\y} &:= {\bf C}\langle \x_{ij}, \y_k, \pd{\x_{ij}}, \pd{\y_k}: 1\leq i\leq j \leq \d, 1\leq k \leq \d \rangle ,\nonumber \\
D_{\x} &:= {\bf C}\langle \x_{ij}, \pd{\x_{ij}}: 1\leq i\leq j \leq \d \rangle,\nonumber\\
D_\t &:= {\bf C}\langle \t_i, \pd{\t_i}: 1\leq i \leq \d\rangle .
\label{eq:notation-for-rings}
\end{align}
We use the following notation.
\begin{equation}
\label{eq:notation}
y = \Sigma^{-1}\mu, \quad x = -\frac{1}{2}\Sigma^{-1}, \quad 
h(x,y,t) = \sum_{i,j=1}^\d\x_{ij}\t_i\t_j+\sum_{i=1}^\d\y_i\t_i.
\end{equation}

\subsection{The case of a smooth function}
We first consider the case when $f(t)$ 
is a smooth function. In this case, the integral $(\ref{eq:general_g})$ converges if the 
matrix $-\x$ is positive definite and 
$f(t)$ is of exponential growth.

To state the main theorem of this section, we show a general lemma.
The notation $f,g,x,y$ in the following lemma is generic and not related to 
\eqref{eq:notation}.

\begin{lemma}
\label{lem:integration_ideal}
Let $D_x$ (resp.\ $D_{xy}$) be the ring of differential operators with
 polynomial coefficient 
 ${\bf C}\langle x_i,\pd{x_i}:1\leq i\leq n\rangle$
 (resp.\ ${\bf C}\langle x_i,y_j, \pd{x_i},\pd{y_j}:1\leq i\le n, 1\le j\le m\rangle$).
Suppose that a holonomic ideal $I$ annihilates a function $f(x,y)$ on 
${\bf R}^n\times{\bf R}^m$ 
and the function $f(x,y)$ is rapidly decreasing with respect to the
 variable $y$ for any $x$ in an open set $O\subset{\bf R}^n$.
Then the integration ideal of $I$ with respect to the variable $y$
annihilates 
\begin{equation}
\label{fn:g-func} 
g(x):= \int_{{\bf R}^m} f(x,y)dy
\end{equation}
defined on the open set $O$.
\end{lemma}
\begin{proof}
Since the function $f(x,y)$ is rapidly decreasing, the integral of
 $(\ref{fn:g-func})$ converges. 
Let $P$ be in the integration ideal $J=\left(I+\sum_{j=1}^m
 \pd{y_j}D_{x,y}\right)\cap D_x $, 
then we have
$
P\bullet g = \int P\bullet f dy
$
by the Lebesgue convergence theorem.
Since the differential operator $P$ can be written as 
$$
P = P_0 + \sum_{i=1}^m \pd{y_i}P_i \in D_x\quad (P_0\in I, \, P_i \in D_{xy}),
$$
we have
$
\int P\bullet f dy = \sum \int \pd{y_i}P_i\bullet f dy.
$
Since $P_i\bullet f$ is rapidly decreasing, we have
$\int \pd{y_i}P_i\bullet f dy = 0$.
\end{proof}

We now go back to $g(x,y)$ in \eqref{eq:general_g}
and use the notation in \eqref{eq:notation-for-rings} and \eqref{eq:notation}.
Consider a ${\bf C}$-algebra morphism $\varphi$ from
$D_\t$ 
to 
$D_{\x\y}$ 
defined by 
\[
\varphi:D_t\rightarrow D_{\x\y}\qquad
\left(
\t_i \mapsto \pd{\y_i},\,
\pd{\t_i}\mapsto -\y_i-2\sum_{k=1}^\d \x_{ik}\pd{\y_k}
\right).
\]
Here, we assume $\x_{ij}=\x_{ji},\, \pd{\x_{ij}}=\pd{\x_{ji}}$.
Since 
$
[\varphi(\pd{\t_i}), \varphi(\pd{\t_j})]
:=\varphi(\pd{\t_i})\varphi(\t_j)-\varphi(\t_j)\varphi(\pd{\t_i})
= \delta_{ij}
$,
where $\delta_{ij}$ is Kronecker's delta, 
$\varphi$ is well-defined as a morphism of ${\bf C}$-algebra.

Now we state the main theorem in this section.
\begin{theorem}\label{thm:E-func}
Suppose a function $f$ on ${\bf R}^\d$ is smooth and of exponential growth.
If differential operators $P_1,\dots, P_s\in D_\t$ annihilate $f$,
then the following differential operators annihilate 
the integral $g(\x, \y)$ in $(\ref{eq:general_g})$.
\begin{eqnarray}
\varphi(P_k) &\quad&(1\leq k \leq s) \label{op:ann-E-1}, \\
\pd{\x_{ij}}-2\pd{\y_i}\pd{\y_j} &\quad&(1\leq i<j\leq \d) \label{op:ann-E-2},\\
\pd{\x_{ii}}-\pd{\y_i}^2 &\quad& (1\leq i \leq \d). \label{op:ann-E-3}
\end{eqnarray}
Moreover, 
the differential operators $(\ref{op:ann-E-1}),(\ref{op:ann-E-2}),(\ref{op:ann-E-3})$
generate a holonomic ideal in $D_{\x \y}$ 
if the differential operators $P_1, \dots, P_s$ generate a holonomic ideal in $D_\t$.
\end{theorem}
\begin{proof}
For a differential operator
$$
P
=\sum_{\alpha,\beta} c_{\alpha,\beta}
 \t_1^{\alpha_1}\cdots \t_\d^{\alpha_\d}
 \pd{\t_1}^{\beta_1}\cdots \pd{\t_\d}^{\beta_\d}
\in D_\t 
$$
and 
$p_i, q_i \in D_{\x \y \t}\,(i=1,\dots, \d)$,
we put
\begin{equation}\label{abbreviation}
P(p_i;q_i)=
\sum_{\alpha,\beta} c_{\alpha,\beta}
 p_1^{\alpha_1}\cdots p_\d^{\alpha_\d}
 q_1^{\beta_1}\cdots q_\d^{\beta_\d}.
\end{equation}
By the assumption, the differential operators
\[
P_\ell \, (\ell = 1,\dots, s),\quad
\pd{\x_{ij}} \, (1\leq i\leq j \leq \d),\quad
\pd{\y_i} \, (1\leq i \leq \d)
\]
annihilate  $f$  and generate a holonomic ideal in 
$D_{\x \y \t}.$
By the lemma introduced in \cite[Section 3.3]{OST}, the differential operators
\[
P_\ell(\t_i;\pd{\t_i}-\frac{\partial h}{\partial \t_i})\, (\ell = 1,\dots,
s),\quad
\pd{\x_{ij}}-\frac{\partial h}{\partial \x_{ij}} \, (1\leq i\leq j \leq \d),\quad
\pd{\y_i}-\frac{\partial h}{\partial \y_i} \, (1\leq i \leq \d)
\]
annihilate $\exp(h)f$ and generate holonomic ideal $I$ in
$D_{\x \y \t}$.
As we will show in Lemma  \ref{lem:genJ} below,
the differential operators
 $(\ref{op:ann-E-1}),(\ref{op:ann-E-2}),(\ref{op:ann-E-3})$ 
 generate the integration ideal $J$ of $I$ with respect to the variables 
 $\t_i\,(i=1,\dots, \d)$.

Since the function $\exp\left( h \right)f$ is rapidly decreasing, the
 integration ideal $J$ annihilates the integral $g(\x, \y)$ by Lemma
 $\ref{lem:integration_ideal}$. Hence, the differential operators
 $(\ref{op:ann-E-1})$,$(\ref{op:ann-E-2})$,\\
 $(\ref{op:ann-E-3})$ annihilate
 $g(\x,\y)$. 

Since the integration ideal of a holonomic ideal is also holonomic
(see, e.g., \cite[Chap 1]{Bjork}),
the ideal $J$ is holonomic when the ideal $I$ is holonomic.
\end{proof}

Now, we show that the integration ideal $J$ is generated by the
 differential operators 
 $(\ref{op:ann-E-1}),(\ref{op:ann-E-2}),(\ref{op:ann-E-3})$ in the following two lemmas.
We denote $P\equiv Q$ when 
$
P-Q \in  
\sum_{i=1}^\d D_{\x \y \t}(\pd{\y_i}-\t_i).
$
\begin{lemma}
 \label{lem:mod1}
Let 
$
p_i =  \pd{\t_i} -\y_i-2\sum_{k=i}^\d\x_{ik}\t_i
$
and 
$
q_i =  \pd{\t_i} -\y_i-2\sum_{k=i}^\d\x_{ik}\pd{\y_k}.
$
For a differential operator 
$P \in D_\t$ and 
a multi index $\alpha \in {\bf N}_0^\d=\{0,1,2\dots\}^\d$, 
the following equivalence relations hold.
\begin{eqnarray}
\label{eq:equiv1}
P(\t_i;p_i)&\equiv& P(\pd{\y_i};q_i),\\
\label{eq:equiv2}
\t^\alpha P(\pd{\y_i};q_i)
&\equiv&
\pd{\y}^\alpha P(\pd{\y_i};q_i).
\end{eqnarray}
Here, we use the notation in $(\ref{abbreviation})$.
\end{lemma}
\begin{proof}
By the straightforward calculation, we have the following relations
for $1\leq i, j \leq \d$.
\begin{eqnarray}
\label{eq:equiv3}
&&
[p_i, q_j] =  0, \\
&&
\label{eq:equiv4}
[p_i, p_j] = [q_i, q_j] = 0, \\ 
&&
\label{eq:equiv5}
[q_i, \t_j] = [q_i,\pd{\y_j}]=\delta_{ij}.
\end{eqnarray}

In order to prove $(\ref{eq:equiv1})$,
it is sufficient to show 
\begin{equation}
\label{eq:equiv8}
\t_1^{\alpha_1}\cdots \t_\d^{\alpha_\d}p_1^{\beta_1} \cdots  p_\d^{\beta_\d} 
\equiv
\pd{\y_1}^{\alpha_1}\cdots \pd{\y_\d}^{\alpha_\d}q_1^{\beta_1}
\cdots  q_\d^{\beta_\d} 
\quad (\alpha,\beta \in {\bf N}_0^\d).
\end{equation}

By the induction on $\beta_i$, we  have a relation
\begin{equation}
\label{eq:equiv6}
\t_jq_i^{\beta_i}
\equiv
\pd{\y_j}q_i^{\beta_i}
\quad (1\leq i, j \leq \d, \quad  \beta_i \in {\bf N}_0).
\end{equation}
When $\beta_i=0$, the relation $(\ref{eq:equiv6})$ holds
 clearly.
Suppose that the relation $(\ref{eq:equiv6})$ holds for $\beta_i$. Then we
 have 
\begin{eqnarray*}
 \t_jq_i^{\beta_i+1}
     &=&  \t_jq_iq_i^{\beta_i}
      =  (q_i\t_j-\delta_{ij})q_i^{\beta_i}\\ 
&\equiv& (q_i\pd{\y_j}-\delta_{ij})q_i^{\beta_i} 
     =  \pd{\y_j}q_iq_i^{\beta_i}
     =  \pd{\y_j}q_i^{\beta_i+1}.
\end{eqnarray*}
By induction, the relation $(\ref{eq:equiv6})$ holds for any $\beta_i$.

By the relations $(\ref{eq:equiv3})$ and $(\ref{eq:equiv6})$,
we have 
\begin{equation}
\label{eq:equiv7}
p_1^{\beta_1} \cdots  p_\d^{\beta_\d}
\equiv
q_1^{\beta_1} \cdots  q_\d^{\beta_\d}
\quad (\beta_i \in {\bf N}_0).
\end{equation}
Now we prove the relation $(\ref{eq:equiv8})$
by induction on the multi index $\alpha\in {\bf N}_0^\d$.
When $\alpha = 0$, the relation $(\ref{eq:equiv8})$ holds because of
 $(\ref{eq:equiv7})$.
Suppose the relation $(\ref{eq:equiv8})$ holds for $\alpha$, then we have
\begin{align*}
\t_i\t_1^{\alpha_1}\cdots \t_\d^{\alpha_\d}p_1^{\beta_1} \cdots  p_\d^{\beta_\d}
&\equiv
\t_i\pd{\y_1}^{\alpha_1}\cdots \pd{\y_\d}^{\alpha_\d}q_1^{\beta_1}
\cdots  q_\d^{\beta_\d} 
\\
&=
\pd{\y_1}^{\alpha_1}\cdots \pd{\y_\d}^{\alpha_\d}
(\prod_{j\neq i}q_j^{\beta_j})\t_iq_i^{\beta_i}
\quad(\text{by $(\ref{eq:equiv4})$ and $(\ref{eq:equiv5})$})\\
&\equiv
\pd{\y_1}^{\alpha_1}\cdots \pd{\y_\d}^{\alpha_\d}
(\prod_{j\neq i}q_j^{\beta_j})\pd{\y_i}q_i^{\beta_i} 
\quad(\text{by $(\ref{eq:equiv6})$})\\
&\equiv
\pd{\y_i}\pd{\y_1}^{\alpha_1}\cdots \pd{\y_\d}^{\alpha_\d}q_1^{\beta_1}
\cdots  q_\d^{\beta_\d}
\quad(\text{by $(\ref{eq:equiv4})$ and $(\ref{eq:equiv5})$}).
\end{align*}
By induction, the relation $(\ref{eq:equiv1})$ holds for any $\alpha$.

Finally, the relation $(\ref{eq:equiv2})$ holds by
 $(\ref{eq:equiv4}),(\ref{eq:equiv5})$ and
$(\ref{eq:equiv6})$.
\end{proof}

\begin{lemma}
 \label{lem:genJ}
 The integration ideal $J$ is generated by the differential operators\\
 $(\ref{op:ann-E-1}),(\ref{op:ann-E-2}),(\ref{op:ann-E-3})$.
\end{lemma}
\begin{proof}
Note that the ideal $I$ is generated by differential operators 
\begin{align*}
&P_\ell(
 \t_i;
 \pd{\t_i}
 -\y_i-2\sum_{k=i}^\d\x_{ik}\t_i
)
\quad (\ell = 1,\dots, s),
\\&
\pd{\x_{ij}}-2\t_i\t_j \quad (1\leq i< j \leq \d),
\\&
\pd{\x_{ii}}-\t_i^2,\,\, \pd{\y_i}-\t_i\quad (1\leq i\leq \d).
\end{align*}
By $(\ref{eq:equiv1})$, the ideal $I$ is generated by
\begin{align}
\label{op:integrand-1}
&
P_\ell(
 \pd{\y_i};
 \pd{\t_i}
 -\y_i-2\sum_{k=i}^\d\x_{ik}\pd{\y_k}
)
\quad (\ell = 1,\dots, s),
\\
\label{op:integrand-2}
&
\pd{\x_{ij}}-2\pd{\y_i}\pd{\y_j} \quad (1\leq i< j \leq \d),
\\
\label{op:integrand-3}
&
\pd{\x_{ii}}-\pd{\y_i}^2 \quad(1\leq i\leq \d),
\\
&
\label{op:integrand-4}
\pd{\y_i}-\t_i \quad(1\leq i\leq \d).
\end{align}

We denote by $\tilde{J}$ the left ideal generated by 
 $(\ref{op:ann-E-1}),(\ref{op:ann-E-2}),(\ref{op:ann-E-3})$.
Clearly we have $\tilde{J} \subset J$.
If $P$ is a differential operator in $J$, then $P$ can be written as 
\begin{equation}\label{P-1}
P = Q + \sum_{i=1}^d \pd{\t_i}R_i +\sum_{i=1}^d S_i\left(\pd{\y_i}-\t_i\right)
\quad (Q \in I, \, R_i,S_i \in D_{\x \y \t})
\end{equation}
by the definition of integration ideal.
The differential operator $Q$ is written as a linear combination of
 the differential operators (\ref{op:integrand-1})--(\ref{op:integrand-3})
 with $D_{\x \y \t}$ coefficients.
By the second term of the right-hand side of $(\ref{P-1})$, we can
 assume without loss of generality that the variables $\pd{\t_i}$ do not appear in these 
 coefficients.
By the relation $(\ref{eq:equiv2})$ in Lemma $\ref{lem:mod1}$,
we can assume that the coefficient in $Q$ is an element in $D_{\x \y}$. 
Since any differential operator in (\ref{op:integrand-1})--(\ref{op:integrand-3}) is an element of 
$
 \tilde{J} +\sum_{i=1}^\d\pd{\t_i}\cdot D_{\x \y \t},
$
we can assume $Q\in J$.

Consider the equation
\[
P - Q -\sum_{i=1}^d  S_i(\pd{\y_i}-\t_i) =  \sum_{i=1}^d \pd{\t_i}R_i .
\]
We can assume that the variables $\pd{\t_1},\dots,\pd{\t_d}$ do not appear 
in $S_1,\dots,S_d$.
For example, if $t_1\pd{t_1}$ is a term of $S_2$ then we can
replace $S_2$ and $R_1$ to $S_2-t_1\pd{t_1}-1$ and $R_1-t_1(\pd{y_2}-t_2)$ respectively.
The term $t_1\pd{t_1}$ in $S_2$ is removed.
In the same way, we can remove all terms which include $\pd{t_i}$. 

Expanding both sides and comparing the coefficients of $\pd{\t_i}$, we have
\[
P - Q =\sum S_i(\pd{\y_i}-\t_i).
\]
The right-hand side of this equation is an element of the left ideal 
$J' :=\sum_{i=1}^\d D_{\x \y \t}\cdot(\pd{\y_i}-\t_i)$.
Let the weight of $\t_i$ be $1$ and that of other variables  $0$, 
and consider a term order $\prec$ with this weight.
The set 
$
\{\t_i-\pd{\y_i} | 1\leq i \leq \d\}
$ 
is a Gr\"{o}bner basis of $J'$ with the order,
so that the initial term of $P-Q$ has to divide some $\t_i$.
Since $P-Q$ is in $D_{\x \y}$, we have $P-Q = 0$. Thus $P \in J$. 
\end{proof}

\subsection{The case of a non-smooth function}
Next we consider the case when $f(\t)$ is not smooth. 
In this case we consider $f(\t)$ as a distribution in the sense of Schwartz (\cite{Schwartz}).

Let $\Omega$ be a domain defined by
$$
\{ (\x, \y) | \text{$-\x$ is positive definite}\}.
$$
For a tempered distribution $f$ on ${\bf R}^\d$, we can define a
function on $\Omega$ as 
\begin{equation}
\label{fn:g-dist}
g(\x,\y) = \langle f, \exp\left(h(\t,\y,\x)\right)\rangle.
\end{equation}
Since $\exp\left(h(\t,\y,\x)\right)$ is rapidly decreasing with
respect to the variable $\t$ when $-\x$ is positive definite, 
the right-hand side of $(\ref{fn:g-dist})$ is finite. 

A holonomic system for $(\ref{fn:g-dist})$ is given as follows.
\begin{theorem}\label{thm:E-dist}
If differential operators $P_1,\dots, P_s\in D_\t$ annihilate a tempered
 distribution $f$ on ${\bf R}^\d$,
then the differential operators
 $(\ref{op:ann-E-1}),(\ref{op:ann-E-2}),(\ref{op:ann-E-3})$ annihilate
 the function $g(\x, \y)$ in $(\ref{fn:g-dist})$. 
Moreover, the differential operators
 $(\ref{op:ann-E-1}),(\ref{op:ann-E-2}),(\ref{op:ann-E-3})$ generate a
 holonomic ideal in 
 $D_{\x \y}$ if the differential operators $P_1, \dots, P_s$
 generate a holonomic ideal in $D_\t$.
\end{theorem}
\begin{proof}
We only need to prove that the differential operators
 $(\ref{op:ann-E-1}),(\ref{op:ann-E-2}),(\ref{op:ann-E-3})$ annihilate
 $g(\x, \y)$. 
Let $P_\ell=\sum c_{\alpha\beta}t^\alpha\pd{t}^\beta$ and 
$P^*_\ell=\sum c_{\alpha\beta}t^\beta\pd{t}^\alpha$.
We have
\allowdisplaybreaks
\begin{align*}
&
P_\ell(\pd{\y_i};
-\y_i-2\sum_{k=i}^\d \x_{ik}\pd{\y_k})
\langle f, \exp\left(h(\t,\y,\x)\right)\rangle
\\
&\qquad\qquad=
\langle f,
P_\ell(\pd{\y_i};
-\y_i-2\sum_{k=i}^\d \x_{ik}\pd{\y_k})
\exp\left(h(\t,\y,\x)\right)\rangle 
\\
&\qquad\qquad=
\langle f,
P_\ell(\pd{\y_i};-\pd{\t_i})
\exp\left(h(\t,\y,\x)\right)\rangle
\\
&\qquad\qquad=
\langle f,
P^*_\ell(-\pd{\t_i};\pd{\y_i})
\exp\left(h(\t,\y,\x)\right)\rangle
\\
&\qquad\qquad=
\langle f,
P^*_\ell(-\pd{\t_i};\t_i)
\exp\left(h(\t,\y,\x)\right)\rangle
\\
&\qquad\qquad=
\langle P_\ell(\t_i;\pd{\t_i}) f,
\exp\left(h(\t,\y,\x)\right)\rangle
\\
&\qquad\qquad=
0.
\end{align*}
\end{proof}

\section{Holonomic system associated with the orthant probability}\label{sec:orthant}
In this section we specialize $f(t)$ of the last section
to the indicator function of the positive orthant and 
we will construct a Pfaffian system associated with the integral $(\ref{eq:g})$
for the orthant probability.

\subsection{Generators of the holonomic ideal}
At first,
we obtain generators of a holonomic ideal which annihilates 
$(\ref{eq:g})$ by Theorem  $\ref{thm:E-dist}$. 
Let $E$ be the positive orthant in ${\bf R}^\d$ defined by 
\[
\left\{\t= (\t_1,\dots,\t_\d)\in {\bf R}^\d \vert \t_i \geq 0, (i= 1,\dots,\d) \right\},
\]
and $\indicator{E}$ be the indicator function of $E$.
A holonomic ideal which annihilates $\indicator{E}$ is given as follows.
\begin{lemma}\label{lem:ann-E}
The indicator function $\indicator{E}$ is annihilated by the 
following differential operators as a distribution.
\begin{equation}\label{eq:ann-E}
\t_1\pd{\t_1}, \dots, \t_\d\pd{\t_\d} 
\end{equation}
The differential operators $(\ref{eq:ann-E})$ generate a holonomic ideal
 $J$ in the ring $D_t$.
\end{lemma}
\begin{proof}
At first, we show that the differential operators $\t_i\pd{\t_i}$
 annihilates the function $\indicator{E}$.
It suffices to prove for $i=1$.
Let $\varphi (\t)$ be a rapidly decreasing function on ${\bf R}^\d$. 
Then, we have
\[
-\int_0^\infty \pd{\t_1}(\t_1\varphi(\t)) d\t_1 = 0
\]
for any $\t_j \in {\bf R}, \, 2\leq j \leq \d$.
Integrating both sides with respect to the variables $\t_2, \dots, \t_\d$, we have
$
\langle \t_1\pd{\t_1} \indicator{E}, \varphi\rangle = 0.
$
Therefore, the distribution $\t_1\pd{\t_1}\indicator{E}$ is equal to $0$.

Next, we show that the left ideal $J$ is holonomic.
By the Buchberger's criterion, we can show that the  set of 
the differential operators in $(\ref{eq:ann-E})$ is 
a Gr\"{o}bner basis of $J$ with the weight $w=(0,1)$. 
The characteristic variety (see, e.g. \cite{Oaku1}, \cite{SST}) of $J$ is 
$
{\rm ch}(J)
=
\{ (\t,\xi)\in {\bf C}^{2\d} \vert \t_i\xi_i = 0,\, 1\leq i \leq \d \}.
$
The variety ${\rm ch}(J)$ can be decomposed as follows,
$$
\bigcup_{J\subset \{1,\dots, \d\}} \{(\t,\xi) \vert \t_i = 0, \xi_j = 0, i\in
 J ,j \notin J\}.
$$
Since the Krull dimension of the each component is $\d$, we have 
${\rm dim }({\rm ch}(J)) = \d$ and the ideal $J$ is holonomic.
\end{proof}

The function $g(\x, \y)$ in $(\ref{eq:g})$ can be 
written as
$
g(\x, \y)
=
\langle 
\indicator{E}(\t),
\exp\left(  h(\x,\y,\t)  \right)
\rangle.
$
This is a case of Theorem $\ref{thm:E-dist}$ in which the distribution $f$
is $\indicator{E}$.
A holonomic ideal which annihilates $\indicator{E}$ is given in Lemma
$\ref{lem:ann-E}$.
Hence we have the following theorem for $g(\x, \y)$ in $(\ref{eq:g})$.
\begin{theorem}\label{thm:hol-orth}
Differential operators
\begin{align}
&
2\sum_{k=1}^\d \x_{ik}\pd{\y_i}\pd{\y_k}  +\y_i\pd{\y_i} + 1
\quad 
(i = 1,\dots, \d, \quad \x_{ij} = \x_{ji}),
\label{ann:orth1}\\
&
\pd{\x_{ij}} - 2\pd{\y_i}\pd{\y_j} 
\quad
(1\leq i< j \leq \d), \label{ann:orth2}\\
&
\pd{\x_{ii}} - \pd{\y_i}^2
\quad
(1\leq i\leq \d)
\label{ann:orth3}
\end{align}
annihilate the function $g(\x, \y)$ in $(\ref{eq:g})$,
and generate a holonomic ideal $I$ in the ring $D_{xy}$.
\end{theorem}

\subsection{Differential recurrence formula}\label{sec:Pfaff}
Next, we give a Pfaffian system associated with $(\ref{eq:g})$.
In order to write the Pfaffian system, we define new  notation.
For $J\subset [\d]=\{1,\dots,d\}$,
we put
\begin{align*}
h_J(\x,\y,x) 
&=\sum_{i\in J}\sum_{j\in J}\x_{ij}\t_i\t_j + \sum_{k\in J} \y_k\t_k,\\
g_J(\x, \y) 
&= \int_0^\infty \dots \int_0^\infty
\exp\left(  h_J(\x,\y,\t)  \right)
dt_J, 
\end{align*}
where 
$d\t_J=\prod_{j\in J} d\t_j$.
When the set $J$ is empty, we set $g_\emptyset = 1$.
For example, the functions are written as follows when $\d=2$. 
\begin{eqnarray*}
g_{\{1,2\}}(\x, \y) &=& \int_0^\infty \exp\left( 
\t_1^2\x_{11}+2\t_1\t_2\x_{12}+\t_2^2\x_{22}+\y_1\t_1 +\y_2\t_2 
\right) d\t_1d\t_2
\\  &=&  g(\x, \y),
\\
g_{\{1\}}(\x, \y) &=& \int_0^\infty \exp\left(\t_1^2\x_{11}+ \y_1\t_1 \right) d\t_1 , 
\\
g_{\{2\}}(\x, \y) &=& \int_0^\infty \exp\left(\t_2^2\x_{22}+ \y_2\t_2 \right) d\t_2 , 
\\
g_\emptyset             &=& 1.
\end{eqnarray*}

Let $J$ be a subset of $[\d]$ and put 
\begin{eqnarray}
\label{op:Qj}
Q_j &=& -\left(\y_j + 2\sum_{k=1}^\d \x_{jk}\pd{\y_k }\right) 
\quad (j= 1,\dots, \d),\\
\label{op:QjJ}
Q_{j,J} &=& -\left(\y_j + 2\sum_{k\in J} \x_{jk}\pd{\y_k }\right) 
\quad (j\in J).
\end{eqnarray}
\begin{lemma}\label{lem:crsp} 
The following equation holds.
\begin{equation}\label{eq:crsp} 
Q_j g_J = Q_{j,J} g_J= g_{J\backslash \{j\} } \quad (j\in J).
\end{equation}
\end{lemma}
\begin{proof}
Since $g_J$ is constant with respect to $y_\ell\ (\ell \notin J)$, 
we have $Q_j g_J = Q_{j,J} g_J$.
We assume that $J=[\d]$ without loss of generality.
Applying $Q_j$ to the integrand of $g_{[\d]}$, we have
\begin{eqnarray*}
&&-\left(\y_j + 2\sum_{k=1}^\d \x_{jk}\pd{\y_k }\right) \exp\left(h(\x,\y,\t)\right) \\
&=& -\left(\y_j + 2\sum_{k=1}^\d \x_{jk}t_k \right) \exp\left(h(\x,\y,\t)\right) \\
&=& -\pd{t_j} \exp\left(h(x,y,t)\right).
\end{eqnarray*}
Integrating both sides of the equation from $0$ to $\infty$ with respect to $t_j$, we have
$$
Q_j\int_0^{\infty} \exp(h(\x,\y,\t)) dt_j = \exp\left( h_{[\d]\backslash \{j\}}(\x,\y,\t)\right).
$$
Integrating both sides of the equation with respect to the remaining variables, we have the equation $(\ref{eq:crsp})$.
\end{proof}

{\bf Remark}\/.
By Lemma \ref{lem:crsp}, we have
$
Q_iQ_jg
=
g_{[d]\backslash \{i,j\}} 
$
for $i\neq j$
When the vector $y$ is equal to zero, 
this equation can be written as 
\[
\left(
2\x_{ij}
+4\sum_{k=1}^d \x_{ik}\x_{jk}\pd{\x_{kk}}
+2\sum_{1\leq k\neq \ell \leq d} \x_{ik}\x_{j\ell}\pd{\x_{k\ell}}
\right)
g
=
g_{[d]\backslash \{i,j\}} .
\]
By the transformation of parameters with 
$\Sigma = (\sigma_{k\ell}) =-\frac{1}{2}\x^{-1}$, 
we have
\[
\frac{\pd{}}{\pd{} \sigma_{k\ell}^J}
\left(
|\sigma|^{-1/2}g
\right)
=
|\sigma|^{-1/2}g_{[d]\backslash \{i,j\}}.
\]
It can be easily checked that this equation corresponds to the classical Schl\"{a}fli's
formula.

\bigskip
For $J=\{j_1,\dots, j_s\}\subset [\d], \,(j_1<j_2<\dots<j_s)$, we denote 
\begin{eqnarray*}
\x_J &=& (\x_{j_kj_\ell})_{1\leq k,\ell \leq s},\quad
\y_J = (\y_{j_1},\dots, \y_{j_s})^T, \\
\Sigma_J &=& -\frac{1}{2}\x_J^{-1} =(\sigma_{ij}^J),\quad
\mu_J = \Sigma_J\y_J = (\mu_{j_1}^J,\dots,\mu_{j_s}^J).
\end{eqnarray*}
The following differential recurrence formula holds.
\begin{theorem}\label{thm:Pfaff}
For any $J\subset [\d]$.
\begin{eqnarray}
 \label{eq:Pfaff-1}
  \pd{\y_i}g_J
  &=& 
  \begin{cases}
   \mu_i^Jg_J + \sum_{j\in J}\sigma_{ij}^Jg_{J\backslash\{j\}} 
   & i \in J\\
   0
   & i \notin J
  \end{cases} \\ 
 \label{eq:Pfaff-2}
  \pd{\x_{ij}}g_J
  &=&
  \begin{cases}
   2\pd{\y_i}\pd{\y_j}g_J& \{i,j \}\subset J,\, i<j\\
   \pd{\y_i}^2g_J& \{i\}\subset J, \, i=j \\
   0 &  else.
  \end{cases}
\end{eqnarray}
\end{theorem}
\begin{proof}
For $i\notin J$, we have $\pd{\y_i}g_J=0$ since the $g_J$ is constant with
 respect to $\y_i$.
For  $i\in J$ Lemma $\ref{lem:crsp}$ implies 
\begin{eqnarray*}
\mu_i^Jg_J + \sum_{j\in J}\sigma_{ij}^Jg_{J\backslash\{j\}} 
&=&
\mu_i^Jg_J - \sum_{j\in J}\sigma_{ij}^J
\left(\y_j^J + 2\sum_{k\in J} \x_{jk}^J\pd{\y_k } \right)g_J \\
&=&
\left(
\mu_i^J- 
\sum_{j\in J}\sigma_{ij}^J\y_j^J 
- 2\sum_{k,j\in J} \sigma_{ij}^J\x_{jk}^J\pd{\y_k } 
\right)g_J.
\end{eqnarray*}
We have 
$
\sum_{j\in J}\sigma_{ij}^J\y_j^J = \mu_i^J
$
by the relation $\mu^J=\Sigma^J\y^J$
and we also have
$$
-2\sum_{k,j\in J} \sigma_{ij}^J\x_{jk}^J \pd{\y_k}
= \sum_{k\in J} \delta_{ik}\pd{\y_k}
= \pd{\y_i}, 
$$
since $-2\Sigma^J\x^J$ is the identity matrix of size $|J|$.
Hence, the right-hand side of $(\ref{eq:Pfaff-1})$ equals $\pd{\y_i}g_J$.

For $\{i,j\}\not\subset J$, $\pd{\x_{ij}}g_J=0$ since the $g_J$ is constant with
 respect to $\x_{ij}$.
For $\{i,j\}\subset J$
$$
\pd{\x_{i,j}} 
\exp\left(
h_J(\x,\y,\t)
\right)
= 
(2-\delta_{ij})
\pd{\y_i}\pd{\y_j}
\exp\left(
h_J(\x,\y,\t)
\right).
$$
Integrating both sides 
we have the  relation $(\ref{eq:Pfaff-2})$.
\end{proof}

By theorem \ref{thm:Pfaff}, we have differential operators $\pd{x_{ij}}-A_{ij}, \pd{y_i}-A_i$
which annihilate the vector value function
$
G(x,y) = (g_J(x,y))_{j\subset [\d]}.
$
Here, $A_{ij}$ and $A_i$ are $2^d\times 2^d$ matrices with rational function entries.
In the next subsection, we prove that the system of these differential operators
is a Pfaffian system in the meaning of \cite[Section 2]{kn2t}.
Note that the Pfaffian system have no singular point on
\begin{equation}\label{positive-definite}
  \{(\x,\y) \vert 
    \text{$-\x$ is positive definite }
  \}.
\end{equation}

\subsection{Pfaffian system and the holonomic rank}\label{sec:rank}
In this section, we give the holonomic rank of the ideal $I$ generated by
(\ref{ann:orth1})--(\ref{ann:orth3}), and show 
that the differential recurrence formula
$(\ref{eq:Pfaff-1})$ and $(\ref{eq:Pfaff-2})$ in the subsection \ref{sec:Pfaff} 
give a Pfaffian system associated with $(\ref{eq:g})$.

In this subsection, we denote the ring of differential operators
in the variables $x, y$ by $R$.
The holonomic rank of $I$ is the dimension of $R/RI$ as 
a vector space over the field of rational functions ${\bf C}(x,y)$ (see, e.g., \cite{n3ost2}).

At first, we give a lower bound of the holonomic rank.
Since the holonomic rank equals to the dimension of holomorphic solutions of $I \bullet f=0$ at generic points,
we can obtain a lower bound of the holonomic rank by constructing linearly
independent functions annihilated by $I$.

\begin{lemma}\label{lem:lower-bound}
The holonomic rank of the ideal $I$ in Theorem $\ref{thm:hol-orth}$ 
is not less than $2^\d$, i.e., $\rank{I} \geq 2^\d$.
\end{lemma}
\begin{proof}
For a vector 
$
\varepsilon 
= 
(\varepsilon_1,\dots, \varepsilon_\d)\in \left\{\pm 1 \right\}^\d
$,
let $E_\varepsilon$ be an orthant
\[
\left\{ 
\t = (\t_1,\dots, \t_\d) \in {\bf R}^\d \vert \varepsilon_i\t_i > 0 \ (i=1,\dots,\d)\right
\}.
\]
It is enough to show that
the following $2^\d$ functions are linearly independent and 
annihilated by $I$;
\[
g_\varepsilon(\x, \y) 
= 
\int_{E_\varepsilon}
\exp\left(
h(\x, \y, \t)
\right)
d\t.
\]

By an analogous way as in the proof of Lemma $\ref{lem:ann-E}$, 
we can show that
the indicator function of $E_{\varepsilon}$ is annihilated 
by the differential operators in $(\ref{eq:ann-E})$ 
for any $\varepsilon \in \{\pm 1\}^\d$.
Analogously to the proof of Theorem $\ref{thm:E-dist}$,
we can show that the differential operators (\ref{ann:orth1})--(\ref{ann:orth3})
annihilate $g_{\varepsilon}(\x, \y)$.

Let $c_\varepsilon$ be a real number for $\varepsilon \in \{\pm 1\}^\d$, 
and suppose $\sum_\varepsilon c_\varepsilon g_\varepsilon = 0$.
Multiplying
 both sides of the equation by 
 $(2\pi)^{-\d/2}(\det \Sigma)n^{-1/2}\exp(-\frac{1}{2}\mu^t\Sigma^{-1}\mu)$, 
we have
$$
\sum_{\varepsilon \in  \{\pm 1\}^\d} c_\varepsilon \Prob(E_\varepsilon \vert \mu,\Sigma) = 0.
$$
Here, $\Prob(E_\varepsilon \vert \mu,\Sigma)$ is the probability of
 the event $E_\varepsilon$ under the multivariate normal distribution
 $N(\mu, \Sigma)$.
Substituting 
$\mu = t\varepsilon, (\varepsilon \in  \left\{ \pm 1 \right\}^\d)$
and taking a limit $t\rightarrow +\infty$,
we have 
\[
c_\varepsilon = 0
\]
Hence, the functions $g_\varepsilon(\x, \y) $ are linearly independent.
\end{proof}

In order to obtain  an upper bound of the holonomic rank of $I$,
we construct bases of $R/RI$ as a linear space over
${\bf C}(\x,\y)$.
The bases correspond to the functions $g_J$.

For $J\subset [\d]$, 
we put a differential operator
\begin{equation}\label{op:P_J}
P_J = \prod_{j'\in [\d]\backslash J} Q_{j'},
\end{equation}
where  $Q_j$ is the differential operator in $(\ref{op:Qj})$.
Note that the differential operators $Q_j$ commute with each other.
By Lemma $\ref{lem:crsp}$, we have
\begin{equation}\label{eq:PJ}
P_J g = g_J.
\end{equation}
Equation $(\ref{eq:PJ})$ means that the differential operator $P_J$ corresponds to 
the function $g_J$.
For example, 
When $\d = 2$ and $J = \emptyset$, the equation $(\ref{eq:PJ})$ is written as follows.  
\[
\left(\y_1 + 2\x_{11}\pd{\y_1} +2\x_{12}\pd{\y_2} \right)
\left(\y_2 + 2\x_{21}\pd{\y_1} +2\x_{22}\pd{\y_2} \right)
g_{\{1,2\}} 
= 1.
\]

Since the differential operator $Q_j$ commutes with 
$\pd{\x_{ij}}-2\pd{\y_i}\pd{\y_j}\, (1\leq i < j \leq \d)$ and 
$\pd{\x_{ii}}-\pd{\y_i}^2\, (1\leq i \leq \d)$, 
we have the following lemma.
\begin{lemma}\label{lem:mod}
The following formulas hold in $R/RI$.
\begin{eqnarray}
&& \pd{\x_{ij}}P_J = 2\pd{\y_i}\pd{\y_j}P_J
  \quad \left(1\leq i < j \leq \d , J\subset [\d]\right),\\
&& \pd{\x_{ii}}P_J = \pd{\y_i}^2P_J
  \quad \left(1\leq i \leq \d , J\subset [\d]\right).
\end{eqnarray}
\end{lemma}
\begin{proof}
For $1\leq i < j \leq \d$ and $1\leq k \leq \d$, we have
\begin{align*}
\pd{\x_{ij}}Q_k
&=  -\pd{\x_{ij}}\left(\y_k + 2\sum_{\ell=1}^\d \x_{k\ell}\pd{\y_\ell }\right)
=  Q_k\pd{\x_{ij}} -2\delta_{ik}\pd{y_j}-2\delta_{jk}\pd{y_i}, \\
2\pd{\y_i}\pd{\y_j}Q_k
&=  -2\pd{\y_i}\pd{\y_j}\left(\y_k + 2\sum_{\ell=1}^\d \x_{k\ell}\pd{\y_\ell }\right) 
=  Q_k2\pd{\y_i}\pd{\y_j} -2\delta_{ik}\pd{y_j}-2\delta_{jk}\pd{y_i}.
\end{align*}
For $1\leq i \leq \d$ and $1\leq k \leq \d$, we have
\begin{align*}
\pd{\x_{ii}}Q_k
&=  -\pd{\x_{ii}}\left(\y_k + 2\sum_{\ell=1}^\d \x_{k\ell}\pd{\y_\ell }\right) 
=  Q_k\pd{\x_{ii}} -2\delta_{ik}\pd{y_i} , \\
\pd{\y_i}^2Q_k 
&=  -\pd{\y_i}^2\left(\y_k + 2\sum_{\ell=1}^\d \x_{k\ell}\pd{\y_\ell }\right) 
=  Q_k\pd{\y_i}^2-2\delta_{ik}\pd{y_i}.
\end{align*}
Therefore, the differential operator $Q_j$  commutes with 
$\pd{\x_{ij}}-2\pd{\y_i}\pd{\y_j}\, (1\leq i < j \leq \d)$ and 
$\pd{\x_{ii}}-\pd{\y_i}^2\, (1\leq i \leq \d)$.
Then we have
\[
(\pd{\x_{ij}}-2\pd{\y_i}\pd{\y_j})P_J
=P_J(\pd{\x_{ij}}-2\pd{\y_i}\pd{\y_j})
=0
\]
in $R/RI$.
Similarly we have
$
(\pd{\x_{ii}}-\pd{\y_i}^2)P_J =0.
$
\end{proof}

The following lemma corresponds to Lemma $\ref{lem:crsp}$.
\begin{lemma}\label{lem:reduction-op}
With the same notations as in Lemma $\ref{lem:crsp}$,
\begin{equation}\label{eq:reduction-op} 
  Q_j P_J = Q_{j,J} P_J= P_{J\backslash \{j\} } \quad (j\in J,\, J\subset [\d]).
\end{equation}
holds in $R/RI$.
\end{lemma}
\begin{proof}
By definition of $P_J$, it is clear that $Q_j P_J = P_{J\backslash \{j\} }$.
Since
$\pd{y_\ell }Q_\ell  
\in I$,
$\pd{\ell}P_J$ is in $I$ if $\ell \notin J$.
Hence we have
\[
Q_j P_J 
=-\left(\y_j + 2\sum_{k=1}^\d \x_{jk}\pd{\y_k }\right) \prod_{\ell\in [\d]\backslash J} Q_\ell
= Q_{j,J} P_J
\]
in $R/RI$.
Note that $Q_j$ and $\pd{y_k}$ commute if $j\neq k$.
\end{proof}

\begin{theorem}\label{thm:Pfaff-D}
The following relations hold in $R/RI$ for any $J\subset [\d]$.
\begin{eqnarray}
 \label{eq:Pfaff-D-1}
  \pd{\y_i}P_J
  &\equiv& 
  \begin{cases}
   \mu_i^JP_J + \sum_{j\in J}\sigma_{ij}^JP_{J\backslash\{j\}} 
   & i \in J\\
   0
   & i \notin J, 
  \end{cases} \\ 
 \label{eq:Pfaff-D-2}
  \pd{\x_{ij}}P_J
  &\equiv& 
  \begin{cases}
   2\pd{\y_i}\pd{\y_j}P_J& \{i,j \}\subset J,\, i<j\\
   \pd{\y_i}^2P_J& \{i\}\subset J, \, i=j \\
   0 &  else.
  \end{cases}
\end{eqnarray}
\end{theorem}
\begin{proof}
For $i\in J$, we can show the equation $(\ref{eq:Pfaff-D-1})$ 
by Lemma \ref{lem:reduction-op} and 
an analogous calculation as in the proof of (\ref{eq:Pfaff-1}).
For $i\notin J$, the equation (\ref{eq:Pfaff-D-1}) is shown 
as in the proof of Lemma \ref{lem:reduction-op}.
By Lemma $\ref{lem:mod}$, we have the equation $(\ref{eq:Pfaff-D-2})$.
\end{proof}

\begin{corollary} 
\label{cor:rank}
The set of the differential operators $P_J\, (J\subset [\d])$ in
 $(\ref{op:P_J})$ spans the quotient space $R/RI$ as a vector space over
 ${\bf C}(\x, \y)$. 
\end{corollary}

This corollary together with Lemma \ref{lem:lower-bound} establishes  the following theorem.
\begin{theorem}\label{thm:rank}
\[
\rank{I} = 2^\d
\]
\end{theorem}

\section{Numerical experiments}\label{sec:numerical}
In this section we present numerical experiments of our holonomic gradient method for
orthant probabilities. Our experiments show that the holonomic gradient method is very accurate
and fast compared to existing methods.

By theorem \ref{thm:Pfaff}, we have an explicit formula for $\frac{\partial}{\partial t}G(x(t), y(t))$
when $x(t)$ and $y(t)$ are smooth functions.
In order to evaluate $G(x,y)$ at $(x_1, y_1)$, we put
\begin{equation}\label{path}
x(t) = (1-t)x_0+tx_1,\ y(t)=ty_1\quad 0\leq t \leq 1.
\end{equation}
Here, $x_0$ is the diagonal matrix whose $(i,i)$-entry equals to that of $x_1$.
The initial value can be written as 
\begin{equation}
 \label{eq:init_val}
  g_J(\x(0), \y(0)) 
=\prod_{j\in J} \left(-\frac{\pi}{4}\x_{jj}(0) \right)^{\frac{1}{2}} .
\end{equation}
Note that $\frac{\partial}{\partial t}G(x(t),y(t))$ does not have singular points on $[0,1]$ 
since the Pfaffian system for $G(x,y)$ does not have singular point on \eqref{positive-definite}.

The accuracy of the holonomic gradient method can be checked by looking at
the summation 
$
\sum_{\varepsilon\in \{\pm 1\}^\d} {\bf P}(X\in E_\varepsilon) = 1
$.
Table $\ref{tab:sum}$ shows errors
$|1-\sum_{\varepsilon\in \{\pm 1\}^\d} {\bf P}(X\in E_\varepsilon)|$
for sample data.
\begin{table}[htpp]
\caption{Errors} \label{tab:sum}
\begin{center}
\begin{tabular}{|c|c|l|}
\hline 
No. &dim &        error \\ \hline
1& 2  & 1.760124e-08 \\ \hline 
2& 3  & 5.473549e-08 \\ \hline 
3& 4  & 3.373671e-08 \\ \hline 
4& 5  & 2.265284e-09 \\ \hline 
5& 6  & 1.120033e-08 \\ \hline 
6& 7  & 7.330036e-09 \\ \hline 
7& 8  & 8.705609e-09 \\ \hline 
8& 9  & 2.288549e-09 \\ \hline 
9&10  & 5.024879e-10 \\ \hline 
\end{tabular}
\end{center}
\end{table}

For a correlation matrix $ R= \{ \rho_{ij} \} $ 
with $\rho_{ij} = \rho, i\neq j$,
the orthant probability can be written as a one-dimensional integral.
Table \ref{tab:comp} shows the values of the orthant probability
calculated by the one-dimensional integral  and HGM for certain values of $\rho$.
The differences of the result of HGM and the one-dimensional integral
are less than $10^{-06}$.
\begin{table}[htpp]
\caption{Comparing of one-dimensional integral} \label{tab:comp}
\begin{center}
\begin{tabular}{|c|l|l|l|}
\hline
 rho  & Dunnett                &  HGM               & error\\ \hline
0.0  & 0.00097656249807343551 &  0.000976562500000  &   1.926564e-12\\ \hline
0.1  & 0.0065864743711607976  &  0.006586475124405  &   7.532442e-10\\ \hline
0.25 & 0.026603192349344017   &  0.026603192572268  &   2.229240e-10\\ \hline
0.5  & 0.090909089922689604   &  0.090909086078297  &   3.844393e-09\\ \hline
\end{tabular}
\end{center}
\end{table}

Table \ref{computational times} shows averages of 
computational times of holonomic gradient method and
Miwa's method for 100 sample data.
The mean vectors of the sample data are all zero, 
and the covariance matrices are randomly generated.
Table \ref{computational times} shows that
our holonomic gradient method evaluates 
orthant probabilities faster than Miwa's method, as the dimension becomes
larger.

However,  it should be noted that 
the holonomic gradient method can be very slow
when the mean vector is far away from zero.
When the mean vector is far away from zero and some eigenvalue of 
the covariance matrix is very small, the value of 
the integral \eqref{eq:g} becomes very large 
since the parameter $y$ becomes large.
In such a case, the Runge-Kutta method takes a long time.

\begin{table}[h]
\begin{center}
\caption{Averages of computational times}\label{computational times}
\begin{tabular}{|c|c|c||c|c|c|}
\hline
dim & Miwa & HGM & dim & Miwa & HGM\\ \hline
 5&  0.002 &0.016 &  9&   6.078 &1.050 \\ \hline
 6&  0.011 &0.056 & 10&  60.171 &2.371\\ \hline
 7&  0.080 &0.154 & 11& 671.370 &5.411\\ \hline
 8&  0.664 &0.390 & 12&   -    &13.48\\ \hline
\end{tabular}
\end{center}
\end{table}

\end{document}